\documentclass{amsart}
\usepackage{tikz-cd}
\numberwithin{equation}{section}
\usepackage{amssymb}
\usepackage{hyperref}
\usepackage{mathtools}
\let\cal\mathcal
\def\Ascr{{\cal A}}
\def\Bscr{{\cal B}}
\def\Cscr{{\cal C}}
\def\Dscr{{\cal D}}

\def\Iscr{{\cal I}}

\def\Kscr{{\cal K}}

\def\Mscr{{\cal M}}

\def\Oscr{{\cal O}}
\def\Pscr{{\cal P}}

\def\Tscr{{\cal T}}

\def\Xscr{{\cal X}}

\let\blb\mathbb

\def \PP{{\blb P}}

\def \Gammasf{\mathsf{\Gamma}}

\def\pr{\mathop{\text{pr}}\nolimits}

\def\Lotimes{\overset{L}{\otimes}}

\def\Mod{\operatorname{Mod}}

\def\Qch{\operatorname{Qch}}
\def\coh{\mathop{\text{\upshape{coh}}}}

\def\rad{\operatorname {rad}}

\def\Ext{\operatorname {Ext}}
\def\Hom{\operatorname {Hom}}

\def\End{\operatorname {End}}
\def\RHom{\operatorname {RHom}}

\def\End{\operatorname {End}}

\def\gldim{\operatorname {gl\,dim}}
\def\pdim{\operatorname {p\,dim}}

\def\r{\rightarrow}

\DeclareMathOperator{\HH}{HH}

\newtheorem{lemma}{Lemma}[section]
\newtheorem{proposition}[lemma]{Proposition}
\newtheorem{theorem}[lemma]{Theorem}
\newtheorem{corollary}[lemma]{Corollary}

\theoremstyle{definition}

\newtheorem{definition}[lemma]{Definition}

{

}

\theoremstyle{remark}

\newtheorem{remark}[lemma]{Remark}

\newdimen\uboxsep \uboxsep=1ex
\def\uboxn#1{\vtop to 0pt{\hrule height 0pt depth 0pt\vskip\uboxsep
\hbox to 0pt{\hss #1\hss}\vss}}

\def\uboxs#1{\vbox to 0pt{\vss\hbox to 0pt{\hss #1\hss}
\vskip\uboxsep\hrule height 0pt depth 0pt}}

\let\oldmarginpar\marginpar
\long\def\marginpar#1{\oldmarginpar{\raggedright \tiny \baselineskip 5pt #1}}

\definecolor{ruta2}{rgb}{0.409, 0.459, 0.208}
\usepackage{xcolor}
\hypersetup{
    colorlinks,
    linkcolor={red!50!black},
    citecolor={blue!50!black},
    urlcolor={blue!80!black}
}

\def\Ob{\operatorname{Ob}}
\def\Perf{\operatorname{Perf}}

\def\aa{\mathfrak{a}}
\def\bb{\mathfrak{b}}
\def\uPerf{\Pscr\mathrm{erf}} 
\def\Dinf{\operatorname{D}_{\infty}}

\def\cone{\operatorname{cone}}
\title{New examples of non-Fourier-Mukai functors}
\author{Theo Raedschelders}
\address[Theo Raedschelders]{Departement Wiskunde, Vrije Universiteit
  Brussel, Pleinlaan $2$, B-1050 Elsene}
\email{theo.raedschelders@vub.be}
\author{Alice Rizzardo}
\address[Alice Rizzardo]{Department of Mathematical Sciences\\
University of Liverpool\\
Mathematical Sciences Building\\
Liverpool L69 7ZL\\
United Kingdom}
\email{alice.rizzardo@liverpool.ac.uk}
\author{Michel Van den Bergh}
\address[Michel Van den Bergh]{Universiteit Hasselt\\Martelarenlaan 42\\3500 Hasselt\\Belgium}
\email{michel.vandenbergh@uhasselt.be}
\email{michel.van.den.bergh@vub.be}
\thanks{The first author is supported by a postdoctoral fellowship from the Research Foundation - Flanders (FWO). The second author is a Lecturer at the University of Liverpool. She is supported by EPSRC grant EP/N021649/1. The third author is a senior researcher at the Research Foundation - Flanders (FWO). He is supported by the FWO-grant G0D8616N ``Hochschild cohomology and deformation theory of triangulated categories.''}
\begin{document}
\begin{abstract}
In this paper we prove that any  smooth projective variety of dimension $\ge 3$ equipped with a tilting bundle
can serve as the source variety of a non-Fourier-Mukai functor between smooth projective schemes.
\end{abstract}
\maketitle
\section{Introduction}
Throughout we fix a base field $k$ and all constructions are linear over $k$.
The first example of a non-Fourier-Mukai functor (see Remark \ref{rem:fm}) between bounded derived categories
of smooth projective schemes was given in \cite{MR3955712}. The functor constructed by the second and third author is of the form $D^b(\coh(Q))\r D^b(\coh(\PP^4))$ where $Q$ is a three-dimensional
smooth quadric and $\PP^4$ is its ambient projective space.
The construction proceeds in two steps. First a prototypical non-Fourier-Mukai functor is constructed between
certain non-geometric DG-categories. Then, using a quite involved argument, this functor is 
turned into a  geometric one.

In this paper we show that if one is not interested in ``small'' examples
the second part of the construction can be simplified by combining results from \cite{MR3545926} with ideas from \cite{Orlov1}.

Recall that if $X$ is a scheme then
a \emph{tilting bundle $T$} on $X$ is a vector bundle on~$X$ such that $\Ext^{>0}_X(T,T)=0$ and such that $T$ generates $D_{\Qch}(\Oscr_X)$.
The following is our main result.
\begin{theorem}[see \S\ref{sec:proof}]
\label{th:mainth}
Let $X$ be a smooth projective scheme of dimension $m\ge 3$ which has a tilting bundle.
Then there is a non-Fourier Mukai functor
\begin{equation}
\label{eq:nonfm}
D^b(\coh(X))\r D^b(\coh(Y))
\end{equation}
where $Y$ is a smooth projective scheme.
\end{theorem}
To apply this theorem we may for example take 
 $X=\PP^m$, $m\ge 3$ which has the Beilinson tilting bundle $\bigoplus_{i=0}^m \Oscr_X(i)$.

\section{Preliminaries on $A_\infty$-categories.}
\label{sec:prelims}
Our general reference for $A_\infty$-algebras and $A_\infty$-categories will be \cite{Lefevre}. Sometimes we silently use notions for categories which are only introduced for algebras (i.e.\ categories with one object) in loc.\ cit.
 We assume that all $A_\infty$-notions are \emph{strictly unital}. Unless otherwise specified we use cohomological grading.
\begin{remark}
\label{rem:quasi}
Below we will rely throughout on the fact that the homotopy categories of $A_\infty$-categories and DG-categories are equivalent.
See \cite{COS}. This implies in particular that we can freely use Orlov's gluing results in \cite{MR3545926} in the $A_\infty$-context.
\end{remark}
\begin{definition}\label{def:non-FM}
Let $\aa$, $\bb$ be pretriangulated $A_\infty$-categories  \cite{BLM}  and put $\Ascr=H^0(\aa)$, $\Bscr=H^0(\bb)$.
We say that an exact
functor $F:\Ascr\r \Bscr$ is  \emph{Fourier-Mukai} if
there is an $A_\infty$-functor $f:\aa\r \bb$ such that
$F\cong H^0(f)$ as graded functors. 
\end{definition}
Often $\aa$, $\bb$ are uniquely  determined by $\Ascr$, $\Bscr$ (see \cite{MR3861804,OrlovLunts}) or else implicit from the context, and then we do not specify them.
\begin{remark}
\label{rem:fm} If $X$, $Y$ are smooth projective varieties and $F:D^b(\coh(X))\r D^b(\coh(Y))$ is a traditional
Fourier-Mukai functor which means that it can be written as  $R\pr_{2\ast}(\Kscr\Lotimes_{X\times Y}L\pr_1^\ast(-))$  for $\Kscr\in D^b(\coh(X\times Y))$ then it is Fourier-Mukai in our sense.
This follows from the easy part of \cite[Theorem 8.15]{Toen} combined with Remark \ref{rem:quasi}.
\end{remark}

For an
$A_\infty$-category $\aa$ we denote by\footnote{$\Dscr_\infty(\aa)$ is
  denoted by $\Cscr_\infty(\aa)$ in \cite{Lefevre}.}
$\Dscr_\infty(\aa)$ the DG(!)-category of left $A_\infty$-modules. The
$A_\infty$-Yoneda functor
\begin{equation}
\label{eq:yoneda}
\aa\r \Dscr_\infty(\aa^\circ):X\mapsto \aa(-,X)
\end{equation}
is quasi-fully faithful \cite[Lemma 7.4.0.1]{Lefevre}. The corresponding homotopy category
$\Dinf(\aa)\coloneqq H^0(\Dscr_\infty(\aa))$ is a compactly generated
triangulated category \cite[\S4.9]{MR2258042} with compact generators $\aa(X,-)$
for $X\in \Ob(\aa)$. 
We write $\uPerf(\aa)$ for the full DG-subcategory
of $\Dscr_\infty(\aa)$  spanned by
the compact objects in $\Dinf(\aa)$ and we also put $\Perf(\aa)=H^0(\uPerf(\aa))$.

If $\Ascr$ is a triangulated category and $S\subset \Ob(\Ascr)$ then
the category \emph{classically generated by $S$} \cite[\S1]{BondalVdB} is the smallest thick
subcategory of $\Ascr$ containing $S$. It is denoted by $\langle S\rangle $.
 By
\cite[\S5.3]{Keller94},\cite[Lemma 2.2]{Neeman92A} $\Perf(\aa)$ is
classically generated by the objects $\aa(X,-)$. 

\medskip

If $f:\aa\r \bb$ is an $A_\infty$-functor then we may view $\bb$ as an $A_\infty$-$\bb$-$\aa$-bimodule. Hence we have a ``standard'' DG-functor
\[
\mathfrak{b} \overset{\infty}{\otimes}_{\mathfrak{a}} -: \Dscr_{\infty}(\mathfrak{a}) \to \Dscr_{\infty}(\mathfrak{b})
\]
which (for algebras) is introduced in \cite[\S4.1.1]{Lefevre}.
We recall the following basic result.
\begin{lemma}
\label{lem:ff}
For $A_{\infty}$-categories $\mathfrak{a}, \mathfrak{b}$ and a
quasi-fully faithful $A_{\infty}$-functor
$f:\mathfrak{a} \to \mathfrak{b}$, the induced functor
$\bb\overset{\infty}{\otimes}_{\aa}-:\Dinf(\aa)\r \Dinf(\bb)$ is fully faithful. Moreover this
functor restricts to a fully faithful Fourier-Mukai functor
$\Perf(\aa)\r \Perf(\bb)$.
\end{lemma}
\begin{proof}
By the  same argument as in the proof of \cite[Lemme 4.1.1.6]{Lefevre} 
there is a quasi-isomorphism 
\begin{equation}
\label{eq:perf}
\mathfrak{b} \overset{\infty}{\otimes}_{\mathfrak{a}} \mathfrak{a}(X,-) \r \mathfrak{b}(fX,-)
\end{equation}
 for $X \in \Ob(\mathfrak{a})$, functorial in $X$.
In other words there is a pseudo commutative diagram
\[
\begin{tikzcd}
H^0(\mathfrak{a})^\circ \ar[r,"H^0(f)"]\ar[d] & H^0(\mathfrak{b})^\circ\ar{d} \\
\Dinf(\mathfrak{a})\ar[r, "\mathfrak{b}\overset{\infty}{\otimes}_{\mathfrak{a}} -"']& \Dinf(\mathfrak{b})
\end{tikzcd}
\]
where the vertical arrows are the Yoneda embeddings $X\mapsto \mathfrak{a}(X,-)$, $Y\mapsto \mathfrak{b}(Y,-)$. 
The full faithfulness of the lower arrow follows by d\'evissage. The claim about $\Perf$ follows immediately from \eqref{eq:perf}.
\end{proof}
The following lemma is a variant on Lemma \ref{lem:ff} and could have been deduced from it.
\begin{lemma} 
\label{lem:porta}
  Assume that $\aa$ is a pre-triangulated $A_\infty$-category \cite{BLM} 
such 
  that $H^0(\aa)$ is Karoubian and classically generated by $T\in \Ob(\aa)$. Put $\mathsf{R}=\aa(T,T)$.
The $A_\infty$-functor
\[
f:
\aa\r \Dscr_\infty(\mathsf{R}^\circ):X\mapsto \aa(T,X)
\]
defines a quasi-equivalence
\[
\aa\r \uPerf(\mathsf{R}^\circ)
\]
or, equivalently, an equivalence of triangulated categories
\begin{equation}
\label{eq:toprove}
H^0(\aa)\cong \Perf(\mathsf{R}^\circ)
\end{equation}
\end{lemma}
\begin{proof}
We must  prove \eqref{eq:toprove}. We have $H^0(f)(T)=\mathsf{R}$.
  By hypothesis $H^0(\aa)$ is classically generated by $T$ and by the
  previous discussion $\Perf(\mathsf{R}^\circ)$ is classically generated by
  $\mathsf{R}$. Moreover since the Yoneda functor is quasi-fully
  faithful, $H^0(f)$ is fully faithful when restricted to~$T$. The
rest follows by d\'evissage.
\end{proof}
\section{Geometric realization of a filtered $A_{\infty}$-algebra}
\label{sec:filtered}
Let $(\mathsf{R},m_*)$ denote a finite-dimensional
$A_{\infty}$-algebra equipped with a (decreasing) filtration
$F^*\coloneqq \{F^p\mathsf{R}\}_{p \geq 0}$.  This means that
$\{F^p\mathsf{R}\}_{p \geq 0}$ is a decreasing filtration of the
underlying graded vector space of $\mathsf{R}$ satisfying the
compatibility conditions
\begin{equation}
\label{eq:filt}
m_p(F^{i_1} \otimes \cdots \otimes F^{i_p}) \subset F^{i_1+\cdots+i_p}
\end{equation}
for all $p$ and all $i_1,\ldots,i_p$. 

Assume $F^n\mathsf{R}=F^n=0$ for some $n\ge 0$.  In this case we may define the \emph{(modified) Auslander $A_{\infty}$-category} $\mathsf{\Gamma}=\mathsf{\Gamma}_{\mathsf{R},F^*}$ of $(\mathsf{R},F^*)$.
The objects of $\Gammasf$ are the
integers $0,\ldots,n-1$ and we set 
\begin{equation}
\mathsf{\Gamma}(j,i)\coloneqq F^{\max(j-i,0)}/F^{n-i}.
\end{equation}
By setting $\Gammasf_{i,j}=\Gammasf(j,i)$, we can represent $\Gammasf$ schematically via the matrix
\begin{equation}
\label{eq:Gamma}
(\mathsf{\Gamma}_{i,j})=
\begin{pmatrix}
\mathsf{R} & F^1 & F^2 & \cdots & F^{n-1} \\[8pt]
\mathsf{R}/F^{n-1} & \mathsf{R}/F^{n-1} & F^1/F^{n-1} & \cdots & F^{n-2}/F^{n-1} \\[8pt]
\mathsf{R}/F^{n-2} & \mathsf{R}/F^{n-2} & \mathsf{R}/F^{n-2} & \cdots & F^{n-3}/F^{n-2} \\[8pt]
\vdots & \vdots & \vdots & \ddots & \vdots \\[8pt]
\mathsf{R}/F^1 & \mathsf{R}/F^1 & \mathsf{R}/F^1 & \mathsf{R}/F^1 & \mathsf{R}/F^1 
\end{pmatrix}
\end{equation}
so that composition is given by matrix multiplication.

The grading on $\mathsf{R}$ induces a grading on $\Gammasf$. Because of condition \eqref{eq:filt}, the higher multiplications on $\mathsf{R}$ also induce multiplications on  $\mathsf{\Gamma}$.  Indeed, 
\begin{equation}
\max(i_{p+1}-i_1,0) \leq \max(i_2-i_1,0) + \cdots + \max(i_{p+1}-i_{p},0),
\end{equation}
so 
\begin{equation}
m_p(F^{\max(i_2-i_1,0)} \otimes \cdots \otimes F^{\max(i_{p+1}-i_p,0)}) \subset F^{\max(i_{p+1}-i_1,0)}.
\end{equation}
Also, 
\begin{multline}
\max(i_2-i_1,0)+\cdots+\max(i_k -i_{k-1})+(n-i_k)\\
+\max(i_{k+2}-i_{k+1},0) + \cdots + \max(i_{p+1}-i_p,0) \\
\geq \max(i_2-i_1,0)+\cdots+\max(i_{k-1} -i_{k-2})+(n-i_{k-1}) 
\geq n -i_1,
\end{multline}
so $m_p$ passes to the quotients 
\begin{equation}
m^{\mathsf{\Gamma}}_p:\mathsf{\Gamma}_{i_1,i_2} \otimes \mathsf{\Gamma}_{i_2,i_3} \otimes \cdots \otimes \mathsf{\Gamma}_{i_{p-1},i_p} \otimes \mathsf{\Gamma}_{i_{p},i_{p+1}} \to \mathsf{\Gamma}_{i_1,i_{p+1}}
\end{equation}
making $\mathsf{\Gamma}$ into an $A_{\infty}$-category. 
\begin{remark}
The same construction also yields the $A_\infty$-algebra $ \bigoplus_{i,j} \Gammasf_{i,j}$, which encodes the same data as $\Gammasf$.
The above construction is similar in spirit to \cite[\S 5]{MR3439086}.
If $\mathsf{R}$ is concentrated in degree $0$ and $F$ is the radical filtration, we obtain a subalgebra of Auslander's original definition \cite{auslander1999representation}.
\end{remark}
\label{sec:filtered2}

Since $\mathsf{\Gamma_{0,0}}=\mathsf{R}$, by thinking of $\mathsf{R}$ as an 
$A_\infty$-category with one object we have a fully faithful strict $A_\infty$-functor
\[
\mathsf{R}\r \Gammasf
\]
whence we obtain by Lemma \ref{lem:ff}:
\begin{corollary}
\label{cor:ff}
There is a fully faithful functor
\[
\Gammasf\overset{\infty}{\otimes}_{\mathsf{R}}-: \Perf(\mathsf{R})\r \Perf(\mathsf{\Gamma}).
\]
\end{corollary}

\begin{proposition}
\label{prop:sod}
Let $\bar{\mathsf{R}}=R/F^1$. There are $n$ quasi-fully-faithful $A_\infty$-functors
\[
\uPerf(\bar{\mathsf{R}})\to \uPerf(\Gammasf)
\]
giving rise to a semi-orthogonal decomposition
\[
\Perf(\mathsf{\Gamma})=\langle\underbrace{ \Perf(\bar{\mathsf{R}}),\ldots,\Perf(\bar{\mathsf{R}})}_n\rangle.
\]
\end{proposition}
\begin{proof} For $i=0,\ldots,n-1$ let
\[
P_i=\Gammasf(i,-)
\]
and $P_n=0$. For $i=0,\ldots,n-1$ the element $P_i \in \Dinf(\Gammasf)$  corresponds to the $i+1^{\text{th}}$ column in \eqref{eq:Gamma} and we have obvious inclusion maps
\[
\psi_i:P_{i+1}\r P_i.
\]
Put 
\begin{equation}
\label{eq:Si}
S_i\coloneqq\cone \psi_i=\begin{pmatrix} F^i/F^{i+1}\\ F^{i-1}/F^{i}\\\vdots\\ \mathsf{R}/F^1\\ 0\\\vdots\\0 \end{pmatrix} 
\end{equation}
(in particular $S_{n-1}=P_{n-1}$). By the Yoneda Lemma we see that 
\begin{equation}
\label{eq:orth}
\Hom^\ast_{\Dinf(\Gammasf)}(P_j,S_i)=H^*(S_i(j))=
\begin{cases}
0&\text{if $j>i$}\\
H^\ast(\bar{\mathsf{R}})&\text{if $j=i$}.
\end{cases}
\end{equation}
We also find using the long exact sequence for the distinguished triangle $P_{i+1}\r P_i\r S_i\r$
\begin{equation}
\label{eq:endast}
\End^\ast_{\Dinf(\Gammasf)}(S_i,S_i)=\Hom^\ast_{\Dinf(\Gammasf)}(P_i,S_i)=H^\ast(\bar{\mathsf{R}}).
\end{equation}
We now have by \eqref{eq:orth} semi-orthonal decompositions
\[
\langle P_i,\ldots, P_{n-1}\rangle =\langle \langle S_i\rangle,\langle P_{i+1},\ldots, P_{n-1}\rangle\rangle,
\]
which by induction yield a semi-ortogonal decomposition
\[
\Perf(\Gammasf)=\langle \langle S_0\rangle,\ldots ,\langle S_{n-1}\rangle \rangle.
\]
Using \eqref{eq:Si} and the compatibility conditions \eqref{eq:filt} for the filtration $F^*$, we check that the $S_i$ are in fact $A_\infty-\Gammasf-\bar{\mathsf{R}}$-bimodules. Thus we have DG functors
\[
S_i\overset{\infty}{\otimes}_{\bar{\mathsf{R}}} -:\Dscr_\infty(\bar{\mathsf{R}})\r \Dscr_\infty(\Gammasf)
\]
and the corresponding exact functors
\[
S_i\overset{\infty}{\otimes}_{\bar{\mathsf{R}}} -:\Dinf(\bar{\mathsf{R}})\r \Dinf(\Gammasf),
\]
which send $\bar{\mathsf{R}}$ to $S_i$ and therefore are fully faithful by \eqref{eq:endast} and Lemma \ref{lem:ff}.
So they establish equivalences
\[
\Perf(\bar{\mathsf{R}}) \cong \langle S_i\rangle 
\]
finishing the proof.
\end{proof}

Let us call an $A_\infty$-algebra $A$ \emph{geometric} if there is a fully faithful Fourier-Mukai functor $f:\Perf A\hookrightarrow D^b(\coh(X))$ for $X$ a smooth and projective $k$-scheme,
such that in addition $f$ has a left and a right adjoint.
\begin{corollary}[Geometric realization] \label{cor:geom}
Let $\mathsf{R}$ be a finite dimensional $A_\infty$-algebra equipped with a finite descending filtration such that $\mathsf{R}/F^1\mathsf{R}$
is geometric.
Then there exists a fully faithful Fourier-Mukai functor
$\Perf{\mathsf{R}}\hookrightarrow D^b(\coh(X))$ where $X$ is a smooth projective $k$-scheme.
\end{corollary}
\begin{proof} Combining Proposition \ref{prop:sod} with \cite[Theorem 4.15]{MR3545926} we obtain that there exists a fully faithful Fourier-Mukai functor 
\[
\Perf{\Gammasf}\hookrightarrow D^b(\coh(X)),
\]
where $X$ is a smooth projective $k$-scheme. Then we pre-compose this functor with the fully faithful Fourier-Mukai functor 
\[
\Perf{\mathsf{R}}\hookrightarrow \Perf{\Gammasf}
\]
of Corollary \ref{cor:ff}.
\end{proof}
\begin{corollary}
\label{cor:geom2}
Assume $\mathsf{R}$ is an $A_\infty$-algebra such that $H^\ast(\mathsf{R})$ is finite dimensional and concentrated in degrees $\le 0$,
and moreover $H^0(\mathsf{R})$ is geometric. Then there exists a fully faithful Fourier-Mukai functor
$\Perf{\mathsf{R}}\hookrightarrow D^b(\coh(X))$, where $X$ is a smooth projective $k$-scheme.
\end{corollary}
\begin{proof} Without loss of generality we may assume that $\mathsf{R}$ is minimal. We now apply Corollary \ref{cor:geom}
with the filtration $F^p\mathsf{R}=\bigoplus_{i\ge p} \mathsf{R}^{-i}$.
\end{proof}
\begin{remark}
Since $H^0(\mathsf{R})$ is assumed to be a finite dimensional algebra, the following lemma may be helpful for checking geometricity of $H^0(\mathsf{R})$ in order to apply Corollary \ref{cor:geom2}:
\newtheorem*{lemma1}{Lemma}
\begin{lemma1} Assume that $A$ is a finite dimensional $k$-algebra. The following are equivalent:
\begin{enumerate}
\item $A$ is geometric.
\item $A$ is smooth (i.e. $\pdim_{A^e}A<\infty$).
\item $A/\rad A$ is separable over $k$ and $\gldim A<\infty$.
\end{enumerate}
\end{lemma1}
\begin{proof}
\strut
\begin{itemize}
\item[$(1)\Rightarrow (2)$] This is
\cite[Theorem 3.25]{MR3545926}. 
\item[$(2)\Rightarrow (3)$] The fact that  $A/\rad A$ is
separable over $k$ is \cite[Theorem 3.6]{reyes2018twisted}, which in turn comes from a MathOverflow answer by Jeremy Rickard \cite{247352}. Finite global dimension is classical.
\item[$(3)\Rightarrow (1)$] This is \cite[Corollary 5.4]{MR3545926}. \qedhere
\end{itemize}
\end{proof}
\end{remark}

\section{Proof of Theorem \ref{th:mainth}}
\label{sec:proof}
Before we proceed with the proof of the Theorem, let us recall some definitions and notation. Unless specified otherwise, in this section $X$ will denote a quasi-compact separated $k$-scheme.

\begin{definition}
If $M\in D(\Oscr_X)$ then the \emph{Hochschild cohomology} of $M$ is defined as 
\[
\HH^\ast(X,M)\coloneqq\Ext^*_{X\times X}(i_{\Delta,\ast}\Oscr_X,i_{\Delta,\ast}M)
\]
where $i_\Delta:X\r X\times X$ is the diagonal map.
\end{definition}

\begin{definition}
Let $X=\bigcup_{i=1}^n U_i$ be an affine covering. For $I\subset \{1,\ldots,n\}$ let $U_I=\bigcap_{i\in I} U_i$. Let $\Iscr$ be the set $\{I\subset\{1,\ldots,n\}\mid I\neq \emptyset\}$. Then $\Xscr$ is defined to be the category with objects
 $\Iscr$ and $\Hom$-sets
\begin{equation}
\Xscr(I,J)=
\begin{cases}
\Oscr_X(U_J)&\text{$I\subset J$}\\
0&\text{otherwise.}
\end{cases}
\end{equation}
\end{definition}

Roughly this allows to think of $\Mod(\Xscr)$ as the category of presheaves associated to an affine covering of $X$. This construction has many good properties. In particular, it will be important for us that there is a fully faithful embedding
\[
w:D(\Qch(X))\to D(\Xscr)
\]
and that for a quasi-coherent sheaf $M$ on $X$ we have
\begin{equation}
\label{eq:hochhoch}
\HH^*(X,M)\cong \HH^*(\Xscr,W(M)),
\end{equation}
where $W(M)$ is the $\Xscr$-bimodule associated to $M$. For more details on this construction, see \cite[\S9.3]{MR3955712} (alternatively see the introduction in loc.\ cit.). 

We will also need a deformed version of $\Xscr$. We give the definition in this case, but the general construction can be found in \cite[\S11]{MR3955712}.

\begin{definition}
Let  $\Mscr$ be a $k$-central $\Xscr$-bimodule and $\eta\in \HH^n(\Xscr,\Mscr)$. Lift $\eta$ to a Hochschild cocycle, which we will also denote by $\eta$. Let $\tilde{\Xscr}$ be the DG-category $\Xscr\oplus \Sigma^{n-2}\Mscr$ whose objects are the objects of $\Xscr$, morphisms are given by $\Xscr(-,-)\oplus \Sigma^{n-2}\Mscr(-,-)$, and composition is coming from the composition in $\Xscr$ and the action of $\Xscr$ on $\Mscr$.

We define as $\Xscr_{\eta}$ the $A_{\infty}$-category $\tilde{\Xscr}$ with deformed $A_{\infty}$-structure given by 
\[
b_{\Xscr_{\eta}}\coloneqq b_{\tilde{\Xscr}}+\eta
\]
where $b_{(-)}$ denotes the codifferential on the corresponding bar construction giving the $A_{\infty}$-structure, and where we view $\eta$ as a map of degree one $(\Sigma\Xscr)^{\otimes n}\to \Sigma(\Sigma^{n-2}\Mscr)$ and extend it to a map $\eta:(\Sigma \Xscr_{\eta})^{\otimes n}\to \Sigma \Xscr_\eta$ by making the unspecified component zero.
\end{definition}

\begin{lemma}
Let $X$ be a smooth projective scheme of dimension $m\ge 3$ which has a tilting bundle. Let  $M=\omega_X^{\otimes 2}$ and $0\neq \eta\in \HH^{2m}(X,M)\cong k$ (\cite[Lemma 10.6.1]{MR3955712}). 
View $\eta$ as an element of $\HH^*(\Xscr,\Mscr)$, for $\Mscr=W(M)$, via \eqref{eq:hochhoch}.
 Then there exists an exact functor
\[
L:D^b(\coh(X))\to \Dinf(\Xscr_\eta)
\]
which is non-Fourier-Mukai (see Definition \ref{def:non-FM}).
\end{lemma}

\begin{proof}
Consider the functor
\[
L:D^b(\Qch(X)) \r \Dinf(\Xscr_\eta)
\]
constructed in \cite[(11.3)]{MR3955712}\footnote{In loc.\ cit.\ we first replace $\Xscr_\eta$ by its $A_\infty$-quasi-isomorphic (unital) DG-hull $\Xscr^{\text{dg}}_\eta$. This is technically
convenient but not essential for the present discussion (cfr Remark \ref{rem:quasi}). Note in particular that since $\Xscr_\eta\r \Xscr^{\mathrm{dg}}
_\eta$ is an $A_\infty$-quasi-isomorphism we have 
$\Dinf(\Xscr_\eta)\cong \Dinf(\Xscr_\eta^{\mathrm{dg}}) \cong D(\Xscr_\eta^{\mathrm{dg}})$ 
by \cite[Lemme 4.1.3.8]{Lefevre}.}. We claim that 
the composition
\begin{equation}
\label{eq:composition3}
D^b(\coh(X))\hookrightarrow D^b(\Qch(X)) \xrightarrow{L} \Dinf(\Xscr_\eta)
\end{equation}
is a non-Fourier-Mukai functor. 

Let $T$ be a tilting bundle for $X$, $A=\End_X(T)$ and $\Tscr=w(T)$. The distinguished triangle in \cite[Lemma 11.3]{MR3955712} applied to $T$ gives a distinguished triangle
\begin{equation}
\Tscr \xrightarrow{\alpha} L(T)\xrightarrow{\beta} \Sigma^{-2m+2}\Mscr^{-1}\otimes_{\Xscr} \Tscr\r 
\end{equation}
and hence
\[
H^\ast(L(T))=\Tscr\oplus \Sigma^{-2m+2}(\Mscr^{-1}\otimes_{\Xscr} \Tscr).
\]
Moreover, by construction, this isomorphism is compatible with the
$H^\ast(\Xscr_\eta)$ and $A$-actions. In the terminology of \cite[\S
7.2, 7.4]{MR3955712}, $L(T)$ is a colift of
$\Tscr\in \Dinf(\Xscr\otimes_k A)$ to
$\Dinf((\Xscr\otimes_k A)_{\eta\cup 1})=\Dinf(\Xscr_\eta\otimes_k A)$.

Setting $\tilde{\Tscr}\coloneqq L(T)$, the fact that such a colift cannot exist is shown in the second part of the proof of \cite[Lemma 12.4]{MR3955712} (the argument as written is for the case $m=3$, but this part of the proof is exactly the same for a general $m\geq 3$). The proof in loc.\ cit.\ computes an obstruction against the existence of  the colift, which is given by the image of $\eta$ under the characteristic morphism defined in \cite[\S7.4]{MR3955712}; this obstruction is shown to be nonzero.
\end{proof}

\medskip

\begin{proof}[Proof of Theorem \ref{th:mainth}]
Let $\Ascr$ be the smallest thick subcategory of $\Dinf(\Xscr_\eta)$
containing the essential image of $D^b(\coh(X))$ under $L$. It is clear that the corestricted functor
\[
L:D^b(\coh(X))\r \Ascr
\]
is still non-Fourier-Mukai.

Let $\aa$ be the full
sub-DG-category of $\Dscr_\infty(\Xscr_\eta)$ spanned by $\Ob(\Ascr)$.
Then we have $H^0(\aa)=\Ascr$.
Let $\mathsf{R}=\aa(L(T),L(T))$. By Lemma \ref{lem:porta} we have a quasi-equivalence
$\aa\r \uPerf(\mathsf{R}^\circ)$. The composed functor
\begin{equation}
\label{eq:composition}
D^b(\coh(X))\xrightarrow{L} \Ascr\xrightarrow{\cong} \Perf(\mathsf{R}^\circ)
\end{equation}
is still non-Fourier-Mukai
since quasi-equivalences are invertible up to homotopy
\cite[Th\'eor\`eme 9.2.0.4]{Lefevre}.

Let $T$ be a tilting bundle for $X$, let $\Tscr=w(T)$ be the left $\Xscr$-module associated to $T$ and let $\Mscr=W(M)$ be the $\Xscr$-bimodule
associated to $M$.
By the discussion before \cite[(12.5)]{MR3955712} we have a distinguished triangle of complexes of vector spaces (taking into account that in the current setting the quantity $n$ in loc.\ cit.\ is equal to $2m$)
\[
\RHom_{\Xscr}(\Sigma^{-2m+2} \Mscr^{-1}\otimes_{\Xscr} \Tscr, \Tscr)\r \RHom_{\Xscr_\eta}(L(T),L(T))\r \RHom_\Xscr(\Tscr,\Tscr)\r 
\]
Using \cite[Lemma 9.4.1]{MR3955712} this becomes
\[
\RHom_{X}(\Sigma^{-2m+2} M^{-1}\otimes_{X} T, T)\r \RHom_{\Xscr_\eta}(L(T),L(T))\r \RHom_X(T,T)\r 
\]
which is equivalent to 
\begin{equation}
\label{eq:distinguished}
\Sigma^{2m-2}\RHom_{X}( T, M\otimes_X T)\r \RHom_{\Xscr_\eta}(L(T),L(T))\r \RHom_X(T,T)\r 
\end{equation}
The cohomology of $\RHom_{X}( T, M\otimes_X T)$ is concentrated in degrees $\le m$. Whence
the cohomology of $\Sigma^{2m-2}\RHom_{X}( T, M\otimes_X T)$ is concentrated in degrees $\le m-(2m-2)<0$ (as $m\ge 3$). 
It now follows from \eqref{eq:distinguished} that $\mathsf{R}$ is an $A_\infty$-algebra such that $H^\ast(\mathsf{R})$ is finite dimensional and concentrated in degrees $\le 0$
and moreover $H^0(\mathsf{R})=\End_X(T)$. As $\End_X(T)^\circ$ is tautologically geometric we obtain by Corollary \ref{cor:geom2}
 a fully faithful Fourier-Mukai functor 
\begin{equation}
\label{eq:geomappl}
\Perf(R^\circ)
\hookrightarrow D^b(\coh(Y)).
\end{equation}
The functor \eqref{eq:nonfm} is now the composition of \eqref{eq:composition} and \eqref{eq:geomappl}. To see that is non-Fourier-Mukai we factor it as 
\begin{equation}
\label{eq:composition2}
D^b(\coh(X))\r \Perf(\mathsf{R}^\circ)\cong \Perf(\mathsf{R}^\circ)\,\tilde{}\subset D^b(\coh(Y))
\end{equation}
where $ \Perf(\mathsf{R}^\circ)\,\tilde{}$ is the essential image of
\eqref{eq:geomappl}. 
Note that since $A_\infty$-quasi-equivalences may be inverted up to homotopy by
\cite[Th\'eor\`eme 9.2.0.4]{Lefevre}, the inverse of
$\Perf(\mathsf{R}^\circ)\cong \Perf(\mathsf{R}^\circ)\,\tilde{}$ is
also a Fourier-Mukai functor. Now if the composition \eqref{eq:composition2} were Fourier-Mukai, then so would be the corestricted functor
$D^b(\coh(X))\r \Perf(\mathsf{R}^\circ)\,\tilde{}$.  Hence the compositon
\[D^b(\coh(X))\r \Perf(\mathsf{R}^\circ)\,\tilde{}\cong
\Perf(\mathsf{R}^\circ)\]
would also be a Fourier-Mukai functor; but this composition is equivalent to \eqref{eq:composition}. This is a
contradiction.
\end{proof}

\begin{remark}
With a little bit more work one may show that the fact that \eqref{eq:composition} is non-Fourier-Mukai is
also true without the hypothesis that $X$ has a tilting bundle. However the tilting bundle is anyway
needed for the rest of the construction.
\end{remark}

\appendix
\section{A different geometrization result}
Our original approach for constructing examples where the conditions
for Corollary \ref{cor:geom} are satisfied was based on Lemma
\ref{lem:original} below. This lemma can be used under some constraints on the shape of the $A_\infty$-algebra, and it provides a different filtration from the one we used in Corollary \ref{cor:geom2}. $A_\infty$-algebras of this shape have been used for example in \cite{MR3955712}. We provide this lemma here since we think it might be potentially useful in other situations.
\begin{lemma} \label{lem:original}
Let $\mathsf{R}$ be a finite dimensional minimal $A_\infty$-algebra such that $\mathsf{R}_i=0$ for $i\not\in \{0,-\kappa\}$ for some $\kappa>0$. Then $\mathsf{R}$ has a finite decreasing filtration as in \S\ref{sec:filtered} such that $\mathsf{R}/F^1\mathsf{R}$ is a semi-simple $k$-algebra.
\end{lemma}
\begin{proof} We consider $\mathsf{R}$ as a $k$-algebra using the multiplication $m_2$, which is associative since $\mathsf{R}$ is minimal.
The $A_\infty$-structure on $\mathsf{R}$ is given by the unique higher multiplication $m_{\kappa+2}:\mathsf{R}_0\times\cdots\times \mathsf{R}_0\r \mathsf{R}_{-\kappa}$. Let $J=\rad\mathsf{R}_0$
and let $a$ be such that $J^a=0$. Let $N\ge (\kappa+2)(a-1)$.
We now give a filtration by graded vector spaces on $\mathsf{R}$
\begin{align*}
F^0\mathsf{R}&=\mathsf{R}_0\oplus \mathsf{R}_{-\kappa}\\
F^1\mathsf{R}&=J\oplus \mathsf{R}_{-\kappa}\\
F^2\mathsf{R}&=J^2\oplus \mathsf{R}_{-\kappa}\\
&\vdots\\
F^{a-1}\mathsf{R}&=J^{a-1}\oplus \mathsf{R}_{-\kappa}\\
F^{a}\mathsf{R}&=J^a\oplus\mathsf{R}_{-\kappa}=\mathsf{R}_{-\kappa}\\
&\vdots\\
F^{N}\mathsf{R}&=J^N\oplus\mathsf{R}_{-\kappa}=\mathsf{R}_{-\kappa}\\
F^{N+1}\mathsf{R}&=J\mathsf{R}_{-\kappa}\oplus\mathsf{R}_{-\kappa} J\\
F^{N+2}\mathsf{R}&=J^2\mathsf{R}_{-\kappa}\oplus J\mathsf{R}_{-\kappa} J\oplus \mathsf{R}_{-\kappa} J^2\\
&\vdots
\end{align*}
We check that this is a filtration of $A_\infty$-algebras. We first check compatibily with~$m_2$. I.e. $m_2(F^p,F^q)\subset F^{p+q}$. The cases $p\ge N$ or $q\ge N$ are clear. So assume $p,q<N$. 
We have for $p,q\le N$: $m_2(F^p,F^q)\subset J^{p+q} \oplus \mathsf{R}_{-\kappa}$. Hence if $p+q\le N$ then there is nothing to prove. So also assume $p+q>N$. Moreover assume $p\le q$ as $q\le p$ is similar. Then we have $m_2(F^p,F^q)\subset J^pR_{-\kappa}\subset F^{p+q}$ (as $q<N$).

To check compatibility with $m_{\kappa+2}$ we have to verify $m_{\kappa+2}(F^{p_1},\cdots,F^{p_{\kappa+2}})\subset F^{\sum_i p_i}$.
If $p_i\ge a$ for some $i$ then the left-hand side is zero and there is nothing to prove. On the other hand, if $p_i\le a-1$ for all $i$ then $\sum_i p_i\le (\kappa+2)(a-1)\le N$.
Hence $m_{\kappa+2}(F^{p_1},\cdots,F^{p_{\kappa+2}})\subset R_{-\kappa} \subset F^{\sum_i p_i}$.
\end{proof}

\bibliography{refs} \bibliographystyle{amsplain}
\end{document}